\documentclass[12pt,a4paper]{article}
\usepackage{enumerate}
\usepackage{amsmath}
\usepackage{amsthm}
\usepackage{amsfonts}
\usepackage{amssymb}
\usepackage[english]{babel}

\numberwithin{equation}{section}
\everymath{\displaystyle}
\title{Viscosity solution PDEs in hybrid games with
mechanical work payoff}

\author{Constantin Udri\c ste, Elena-Laura Otob\^{i}cu, Ionel \c Tevy}
\begin{document}
\date{}
\renewcommand{\abstractname}{}

\maketitle
\numberwithin{equation}{section}
\newtheorem{theorem}{Theorem}[section]
\newtheorem{lemma}[theorem]{Lemma}
\newtheorem{corollary}[theorem]{Corollary}
\newtheorem{definition}[theorem]{Definition}
\newtheorem{remark}[theorem]{Remark}

\begin{abstract}
In a multitime hybrid differential game with mechanical work payoff, 
the multitime upper value function and the multitime lower value function 
are viscosity solutions of original PDEs of type Hamilton-Jacobi-Isaacs.
\end{abstract}

\noindent {\bf MSC2010}: multitime hybrid differential games; multitime viscosity solution; multitime dynamic programming.


\section{Multitime lower or upper value function}

All variables and functions must satisfy suitable conditions (for example, see \cite{[8]}).
We analyze a {\it multitime hybrid differential game, with two teams of players}, whose {\it Bolza payoff} is the sum between a path independent curvilinear integral (mechanical work) and a function of the final event (the terminal cost, penalty term) and whose evolution PDE is an {\it m-flow}. The optimal control problem is:

 {\it Find
$$\min_{v(\cdot)\in V}\max_{u(\cdot)\in U} J(u(\cdot),v(\cdot))=\int_{\Gamma_{0T}} L_\alpha (s,x(s),u_\alpha(s),v_\alpha(s))ds^\alpha+g(x(T)),$$
subject to the Cauchy problem
$$\frac{\partial x^i}{\partial s^\alpha}(s)=X^i_\alpha(s,x(s),u_\alpha(s),v_\alpha(s)),$$
$$x(0)=x_0, \,\,s\in \Omega_{0T}\subset \mathbb{R}_+^m, \,\,x\in \mathbb{R}^n,$$}
where $i=1,...,n$; $\alpha =1,...,m$; $u=(u^a)$, $a=1,...,p$, $v=(v^b)$, $b=1,...,q$ are the controls.

To simplify, suppose that the curve $\Gamma_{0T}$ is an increasing curve in the multitime interval $\Omega_{0T}$.

We vary the starting multitime and the initial point. We obtain a larger family of similar multitime problems containing the functional
$$J_{x,t}(u(\cdot),v(\cdot))=\int_{\Gamma_{tT}} L_\alpha (s,x(s),u_\alpha(s),v_\alpha(s))ds^\alpha+g(x(T)),$$
and the evolution constraint (Cauchy problem for first order PDEs system)
$$\frac{\partial x^i}{\partial s^\alpha}(s)=X^i_\alpha(s,x(s),u_\alpha(s),v_\alpha(s)),\,\, x(t)=x, \,\,s\in \Omega_{tT}\subset \mathbb{R}_+^m,\,\, x\in \mathbb{R}^n.$$

\begin{definition} Let $\Psi$ and $\Phi$ be suitable strategies of the two equips of players.

(i) The function
$$m(t,x)=\min_{\Psi\in \mathcal{V}} \max_{u(\cdot)\in U} J_{t,x}( u(\cdot),\Psi[u](\cdot))$$ is called the multitime lower value function.

(ii) The function
$$M(t,x)=\max_{\Phi\in \mathcal{U}} \min_{v(\cdot)\in V} J_{t,x}(\Phi[v](\cdot),v(\cdot)) $$ is called the multitime upper value function.
\end{definition}

The papers \cite{[1]}-\cite{[4]}, \cite{[12]} refer to viscosity solutions of Hamilton-Jacobi-Isaacs equations.
To understand the multitime optimal control and our recent results see the papers \cite{[5]}-\cite{[11]}.

\section{Viscosity solutions of \\multitime upper/lower PDEs}

The key original idea is that the multitime upper value function or the multitime lower value function
are solutions of PDEs, defined in the next Theorem.
Our PDEs contain some implicit assumptions and are valid under certain conditions
which are defined and analyzed for multitime hybrid differential games.

\begin{theorem}
(i) The multitime upper value function $M(t,x)$ is the viscosity solutions of the multitime upper PDE
$$\frac{\partial M}{\partial t^\alpha}(t,x)+\min_{v_\alpha\in \mathcal{V}} \max_{u_\alpha\in \mathcal{U}} \left\lbrace  \frac{\partial M}{\partial x^i}(t,x) X_\alpha^i(t,x,u_\alpha,v_\alpha)+L_\alpha(t,x,u_\alpha,v_\alpha)\right\rbrace =0,$$
which satisfies the terminal condition $M(T,x)=g(x).$

(ii) The multitime lower value function $m(t,x)$ is the viscosity solution of the multitime lower PDE
$$\frac{\partial m}{\partial t^\alpha}(t,x)+\max_{u_\alpha \in \mathcal{U}} \min_{v_\alpha \in \mathcal{V}} \left\lbrace  \frac{\partial m}{\partial x^i}(t,x) X_\alpha^i(t,x,u_\alpha,v_\alpha)+L_\alpha(t,x,u_\alpha,v_\alpha)\right\rbrace =0,$$
which satisfies the terminal condition $m(T,x)=g(x).$
\end{theorem}

\begin{proof} We introduce the so-called upper and lower Hamiltonian defined respectively by
$$H^+_\alpha(t,x,p)=\min_{v_\alpha\in \mathcal{V}} \max_{u_\alpha \in \mathcal{U}}\lbrace p_i(t) X_\alpha^i(t,x,u_\alpha,v_\alpha)+L_\alpha(t,x,u_\alpha,v_\alpha)\rbrace,$$
$$H^-_\alpha(t,x,p)=\max_{u_\alpha\in \mathcal{U}} \min_{v_\alpha\in \mathcal{V}}\lbrace p_i(t) X_\alpha^i(t,x,u_\alpha,v_\alpha)+L_\alpha(t,x,u_\alpha,v_\alpha)\rbrace.$$

We prove only the first statement. For $s\in \Omega_{tt+h},$ we use the Cauchy problem
$$\frac{\partial x^i}{\partial s^\alpha}(s)=X^i_\alpha(s,x(s),u_\alpha(s),v_\alpha(s)),$$
$$x(t)=x,\,\, s\in \Omega_{tt+h}\subset \mathbb{R}_+^m,\,\, x\in \mathbb{R}^n$$
and the cost functional (mechanical work)
$$J_{t,x}(u(\cdot),v(\cdot))=\int_{\Gamma_{tt+h}} L_\alpha (s,x(s),u_\alpha(s),v_\alpha(s))ds^\alpha.$$
For $s\in \Omega_{tT}\setminus \Omega_{tt+h},$ the cost is $M(t+h,x(t+h))$.
Consequently,
 $$J_{t,x}(u(\cdot),v(\cdot))=\int_{\Gamma_{tt+h}} L_\alpha (s,x(s),u_\alpha(s),v_\alpha(s))ds^\alpha+M(t+h,x(t+h)),$$
with $M(t,x)\geq M(t+h,x(t+h))$, because $M(t,x)$ is the greatest cost.

Thus we have the multitime dynamic programming optimality condition

$$
M(t,x)=\max_{\Phi\in \mathcal{A}(t)}\min_{v_\alpha\in V(t)}\bigg\{\int_{\Gamma_{tt+h}} L_\alpha (s,x(s),\Phi [v_\alpha](s),v_\alpha(s))ds^\alpha +M(t+h,x(t+h))\bigg\}.
$$

Let $(\omega)\in C^1(\Omega_{0T}\times \mathbb{R}^n)$ be a generating vector field.
We analyse two cases:

\medskip
{\bf Case 1} Suppose $M-\omega$ attains a local maximum at $(t,x)\in \Omega_{0T}\times \mathbb{R}^n.$
We must prove the inequality
$$\frac{\partial \omega}{\partial t^\alpha}(t,x)+H^+_{\alpha}\left( t,x,\frac{\partial \omega}{\partial x^i}(t,x)\right) \geq 0.\eqno(1)$$

For that, we suppose the contrary
$$
\frac{\partial \omega}{\partial t^\alpha}(t,x)+H^+_{\alpha}\left( t,x,\frac{\partial \omega}{\partial x^i}(t,x)\right) \leq -\theta_\alpha <0,
$$
for each $\alpha=\overline{1,m}$ and for some constant 1-form $\theta_\alpha>0.$

Let $h=(h^\alpha),$ with $h^\alpha>0.$

We use the {\it Fundamental Lemma} in the next Section. This implies that, for each sufficiently small $\Vert h\Vert$ and all $\omega\in \mathcal{A}(t),$ the relation
$$
\int_{\Gamma_{tt+h}} \bigg(L_\alpha(s,x(s),\Phi[v_\alpha](s),v_\alpha(s)) +\frac{\partial \omega}{\partial x^i} X_\alpha^i(s,x(s),\Phi[v_\alpha](s),v_\alpha(s))+\frac{\partial \omega}{\partial s^\alpha}\bigg) ds^\alpha \leq -\frac{h^{\alpha}\theta_{\alpha}}{2}
$$
holds for $v_\alpha\in \mathcal{V}(t).$
Thus
$$
\max_{\Phi\in \mathcal{A}(t)}\min_{v_\alpha\in V(t)} \bigg\{\int_{\Gamma_{tt+h}} \left(L_\alpha(s,x(s),\Phi[v_\alpha](s),v_\alpha(s)) +\frac{\partial \omega}{\partial x^i} X_\alpha^i(s,x(s),\Phi[v_\alpha](s),v_\alpha(s))\right.
$$
$$
\hspace{5cm}\left.+\frac{\partial \omega}{\partial s^\alpha}\right)ds^\alpha \bigg\}\leq -\frac{h^{\alpha}\theta_{\alpha}}{2},\eqno(2)
$$
with $x(\cdot)$ solution of the previous Cauchy problem.

Because $M-\omega$ has a local maximum at the point $(t,x),$ we have
$$M(t,x)-\omega(t,x)\geq M(t+h,x(t+h))-\omega(t+h,x(t+h)).$$

The multitime dynamic programming optimality condition and by the local maximum definition, we can write
$$M(t,x)-M(t+h,x(t+h))=\max_{\Phi\in \mathcal{A}(t)}\min_{v_\alpha\in \mathcal{V}(t)}  \int_{\Gamma_{tt+h}} L_\alpha (s,x(s),\Phi [v_\alpha](s),v_\alpha(s))ds^\alpha. $$

Consequently, we have
$$\max_{\Phi\in \mathcal{A}(t)}\min_{v_\alpha\in \mathcal{V}(t)}  \int_{\Gamma_{tt+h}} L_\alpha (s,x(s),\Phi [v_\alpha](s),v_\alpha(s))ds^\alpha \geq \omega(t,x)-\omega(t+h,x(t+h))$$
or
$$
\max_{\Phi\in \mathcal{A}(t)}\min_{v_\alpha\in \mathcal{V}(t)}\int_{\Gamma_{tt+h}} L_\alpha (s,x(s),\Phi [v_\alpha](s),v_\alpha(s))ds^\alpha  +\omega(t+h,x(t+h))-\omega(t,x)\geq 0.\eqno(3)
$$

On the other hand,
$$
\omega(t+h,x(t+h))-\omega(t,x)= \int_{\Gamma_{tt+h}} d\omega=\int_{\Gamma_{tt+h}} D_\alpha\omega\,\, ds^\alpha
$$
$$ =\int_{\Gamma_{tt+h}} \left(  \frac{\partial \omega}{\partial x^i} X_\alpha^i(s,x(s),\Phi[v_\alpha](s),v_\alpha(s))+\frac{\partial \omega}{\partial s^\alpha}\right)  ds^\alpha.
$$
So, the relation $(3)$ contradicts the relation $(2)$.

\medskip
{\bf Case 2} Suppose $M-\omega$ attains a local minimum at $(t,x)\in \Omega_{0T}\times \mathbb{R}^n.$
We must prove that
$$\frac{\partial \omega}{\partial t^\alpha}(t,x)+H^+_{\alpha}(t,x,\frac{\partial \omega}{\partial x^i}(t,x))\leq 0, \eqno(4).$$
To do this, we suppose the contrary
$$
\frac{\partial \omega}{\partial t^\alpha}(t,x)+H^+_{\alpha}(t,x,\frac{\partial \omega}{\partial x^i}(t,x))\geq \theta_\alpha >0,
$$
for each $\alpha=\overline{1,m}$ and for some constant 1-form $\theta_\alpha>0.$

Let $h=(h^\alpha),$ with $h^\alpha>0.$

We use the {\it Fundamental Lemma} in the next Section. This implies that, for each sufficiently small $\Vert h\Vert$ and all $\omega\in \mathcal{A}(t),$ the relation
$$
\int_{\Gamma_{tt+h}} \bigg( L_\alpha(s,x(s),\Phi[v_\alpha](s),v_\alpha(s)) +\frac{\partial \omega}{\partial x^i} X_\alpha^i(s,x(s),\Phi[v_\alpha](s),v_\alpha(s)) +\frac{\partial \omega}{\partial s^\alpha}\bigg) ds^\alpha\geq\frac{h^\alpha\theta_\alpha}{2}
$$
holds for $v_\alpha\in \mathcal{V}(t).$
Thus
$$
\max_{\Phi\in \mathcal{A}(t)}\min_{v_\alpha\in \mathcal{V}(t)} \bigg\{ \int_{\Gamma_{tt+h}}  \left( L_\alpha(s,x(s),\Phi[v_\alpha](s),v_\alpha(s)) +\frac{\partial \omega}{\partial x^i} X_\alpha^i(s,x(s),\Phi[v_\alpha](s),v_\alpha(s))\right.
$$
$$
\hspace{5cm}\left.+\frac{\partial \omega}{\partial s^\alpha}\right) ds^\alpha\bigg\}\geq\frac{h^\alpha\theta_\alpha}{2}.\eqno(5)
$$

Because $M-\omega$ has a local minimum at the point $(t,x),$ we have
$$M(t,x)-\omega(t,x)\leq M(t+h,x(t+h))-\omega(t+h,x(t+h)),$$ where $x(\cdot)$
is the solution of the previous Cauchy problem.

By the multitime dynamic programming optimality condition and by
the local minimum definition, we can write
$$M(t,x)-M(t+h,x(t+h))=\max_{\Phi\in \mathcal{A}(t)}\min_{v_\alpha\in \mathcal{V}(t)} \left\lbrace \int_{\Gamma_{tt+h}} L_\alpha (s,x(s),\Phi [v_\alpha](s),v_\alpha(s))ds^\alpha\right\rbrace.$$
Using the inequality
$$M(t,x)-M(t+h,x(t+h))\leq \omega(t,x)-\omega(t+h,x(t+h)),$$
we find
$$\max_{\Phi\in \mathcal{A}(t)}\min_{v_\alpha\in \mathcal{V}(t)} \left\lbrace\int_{\Gamma_{tt+h}} L_\alpha (s,x(s),\Phi [v_\alpha](s),v_\alpha(s))ds^\alpha \right\rbrace\leq \omega(t,x)-\omega(t+h,x(t+h))$$
and
$$
\max_{\Phi\in \mathcal{A}(t)}\min_{v_\alpha\in \mathcal{V}(t)}\bigg\{  \int_{\Gamma_{tt+h}} L_\alpha (s,x(s),\Phi [v_\alpha](s),v_\alpha(s))ds^\alpha\bigg\}  +\omega(t+h,x(t+h))-\omega(t,x)\leq 0.\eqno(6)
$$

On the other hand,
$$
\omega(t+h,x(t+h))-\omega(t,x)=  \int_{\Gamma_{tt+h}} d\omega=\int_{\Gamma_{tt+h}} D_\alpha\omega \,\,ds^\alpha
$$
$$
=\int_{\Gamma_{tt+h}} \bigg(\frac{\partial \omega}{\partial x^i} X_\alpha^i(s,x(s),\Phi[v_\alpha](s),v_\alpha(s))+\frac{\partial \omega}{\partial s^\alpha}\bigg)ds^\alpha.
$$
That is why the relation $(6)$ contradicts the relation $(5).$
\end{proof}

\section{Fundamental contradict Lemma}

The short proofs in the previous section are based on an interesting Lemma. 

\begin{lemma}\label{l-1}
Let $\omega\in C^1(\Omega_{0T}\times \mathbb{R}^n)$.

(i)If $M-\omega$ attains a local maximum at $(t_0,x_0)\in \Omega_{0T}\times \mathbb{R}^n$ and
$$\omega_{t^{\alpha}}(t_0,x_0)+ H^+_\alpha\left(t_0,x_0,\frac{\partial\omega}{\partial x^i}(t_0,x_0)\right)\leq - \theta_\alpha <0,$$
then, for all vectors $h=(h^\alpha)$, with sufficiently small $||h||$,
there exists a control $v=(v_\alpha)\in {\mathcal V}(t_0)$ such that
the relation $(2)$ holds for all strategies $\Phi\in {\mathcal A}(t_0)$.

(ii) If $M-\omega$ attains a
local minimum at $(t_0,x_0)\in \Omega_{0T}\times \mathbb{R}^n$ and
$$\omega_{t^{\alpha}}(t_0,x_0)+ H^+_\alpha\left(t_0,x_0,\frac{\partial\omega}{\partial x^i}(t_0,x_0)\right)\geq \theta_\alpha >0,$$
then, for all vectors $h=(h^\alpha)$, with sufficiently small $||h||$,
there exists a control $u=(u_\alpha)\in {\mathcal U}(t_0)$ such that
the relation $(5)$ holds for all strategies $\Psi\in {\mathcal B}(t_0)$.

\end{lemma}

\begin{proof}
We introduce the 1-form $\Lambda$ of components
$$
\Lambda_\alpha = L_\alpha(s,x(s),\Phi[v_\alpha](s),v_\alpha(s))+\frac{\partial \omega}{\partial x^i} X_\alpha^i(s,x(s),\Phi[v_\alpha](s),v_\alpha(s))+\frac{\partial \omega}{\partial s^\alpha}.
$$

(i) By hypothesis
$$\min_{v\in {\mathcal V}}\,\, \max_{u\in {\mathcal U}} \,\Lambda_\alpha (t_0,x_0,u_\alpha,v_\alpha)\leq -\theta_\alpha <0.$$
Consequently there exists some control $v^*\in {\mathcal V}$ such that
$$\max_{u\in {\mathcal U}} \, \Lambda_\alpha (t_0,x_0,u_\alpha,v^*_\alpha)\leq -\theta_\alpha,$$ for each $\alpha=\overline{1,m}.$
On the other hand, the uniform continuity of the 1-form $\Lambda=(\Lambda_\alpha)$ implies
$$\max_{u\in {\mathcal U}} \, \Lambda_\alpha (t_0,x(s),u_\alpha,v^*_\alpha)\leq - \frac{1}{2}\,\theta_\alpha$$
provided $s\in \Omega_{t_0t_0+h}$, for any small $||h||>0$, and $x(\cdot)$ is solution of PDE on $\Omega_{t_0t_0+h}$,
for any $u(\cdot),v(\cdot)$, with initial condition $x(t_0)=x_0$. It follows that,
for the control $v(\cdot)=v^*$ and for any strategy $\Phi\in {\mathcal A}(t_0)$, we have
$$
L_\alpha(s,x(s),\Phi[v_\alpha](s),v_\alpha(s))+\frac{\partial \omega}{\partial x^i} X_\alpha^i(s,x(s),\Phi[v_\alpha](s),v_\alpha(s))+\frac{\partial \omega}{\partial t^\alpha}\leq \frac{-\theta_\alpha}{2}
$$
for $s\in \Omega_{t_0t_0+h}$. Taking the curvilinear integral along
an increasing curve $\Gamma_{t_0t_0+h}$, we obtain the relation $(2)$.

(ii) The inequality in the Lemma reads
$$\min_{v\in {\mathcal V}}\,\, \max_{u\in {\mathcal U}} \,\Lambda_\alpha (t_0,x_0,u_\alpha,v_\alpha)\geq \theta_\alpha >0.
$$
Consequently, for each control $v\in {\mathcal V}$ there exists a control $u=u(v)\in U$ such that
$$\Lambda_\alpha (t_0,x_0,u_\alpha,v_\alpha)\geq \theta_\alpha.$$
The uniform continuity of the 1-form $\Lambda$ implies
$$\Lambda_\alpha (t_0,x_0,u_\alpha,\xi_\alpha)\geq \frac{3}{4}\,\theta_\alpha,\,\,\forall \xi\in B(v,r)\cap V\, \hbox{and some}\,\, r=r(v)>0.$$
Due to compactness of ${\mathcal V}$, there exists finitely many distinct points
$$v_1,...,v_n \in {\mathcal V}; \,u_1,..., u_n\in {\mathcal U}$$
and the numbers $r_1,...,r_n >0$ such that ${\mathcal V}\subset \bigcup_{i=1}^{n} B(v_i,r_i)$ and
$$\Lambda_\alpha (t_0,x_0,u_i,\xi)\geq \frac{3}{4}\,\theta_\alpha,\,\,\forall \xi\in B(v_i,r_i).$$

Define
$$\psi:{\mathcal V}\to {\mathcal U},\, \psi(v)=u_k\,\, \hbox{if}\,\, v\in B(u_k,r_k)\setminus \bigcup_{i=1}^{k-1} B(u_i,r_i),\,k=\overline{1,n}.$$
In this way,
$$\Lambda_\alpha (t_0,x_0,\psi(v_\alpha),v_\alpha)\geq \frac{3}{4}\,\theta_\alpha, \forall v\in {\mathcal V}.$$
Again, the uniform continuity of the 1-form $\Lambda$ and a sufficiently small $||h||>0$ give
$$\Lambda_\alpha (s,x(s),\psi(v_\alpha),v_\alpha)\geq \frac{1}{2}\,\theta_\alpha, \forall v\in {\mathcal V},\, s\in \Omega_{t_0t_0+h},$$
and any solution $x(\cdot)$ of PDE on $\Omega_{t_0t_0+h}$, for any $u(\cdot), v(\cdot)$ and with initial condition $x(t_0)=x_0$.
Now define a new strategy
$$\Psi\in {\mathcal B}(t_0),\,\,\Psi[v_\alpha](s)=\psi(v_\alpha(s)),\, \forall v\in {\mathcal V}(t_0),\,s\in \Omega_{t_0t_0+h}.$$
Finally, for each $\alpha,$ we have the inequality
$$\Lambda_\alpha (s,x(s),\Psi[v_\alpha](s),v_\alpha(s))\geq \frac{1}{2}\,\theta_\alpha, \forall \,s\in \Omega_{t_0t_0+h},$$
and taking the curvilinear integral along an increasing curve $\Gamma_{t_0t_0+h}$, we find the result in Lemma.
\end{proof}

Constantin Udri\c ste, Elena-Laura Otob\^{i}cu, Ionel \c Tevy\\ 
University {\sc Politehnica} of Bucharest, Faculty of Applied
Sciences, \\Department of Mathematics and Informatics, \\ Splaiul
Independentei 313, RO-060042, Bucharest, Romania.\\
   E-mail: {udriste@mathem.pub.ro}; {laura.otobicu@gmail.com}; {vascatevy@yahoo.fr}

\end{document}